\documentclass[11pt, leqno]{amsart}
\usepackage{amscd,verbatim,amssymb}
\usepackage{color}
\usepackage{verbatim}

\usepackage{hyperref}
\usepackage{xy}
\xyoption{all}
\usepackage{appendix}
\newtheorem{theorem}{Theorem}[section]
\newtheorem{proposition}[theorem]{Proposition}
\newtheorem{lemma}[theorem]{Lemma}
\newtheorem{corollary}[theorem]{Corollary}

\theoremstyle{definition}

\newtheorem{example}[theorem]{Example}
\newtheorem{examples}[theorem]{Examples}
\newtheorem{contrexample}[theorem]{\bf Contrexample}

\newtheorem{dfn}[theorem]{Definition}

\newcommand{\m}{^{\times}}

\newcommand{\inv}{^{-1}}

\newcommand{\id}{\mathrm{id}}

\newcommand{\ind}{\operatorname{ind}}

\newcommand{\res}{\operatorname{res}}
\newcommand{\cor}{\operatorname{cor}}

\newcommand{\Trd}{\operatorname{Trd}}
\newcommand{\ord}{\operatorname{ord}}

\newcommand{\Z}{\mathbb{Z}}

\newcommand{\Q}{\mathbb{Q}}

\newcommand{\cA}{\mathcal A}
\newcommand{\cB}{\mathcal B}

\newcommand{\cH}{\mathcal H}
\newcommand{\cM}{\mathcal M}

\numberwithin{equation}{section}
\title{Power-central elements in tensor products of symbol algebras}
\date{}
\author[D. Barry]{Demba  Barry}
\address{ICTEAM Institute, Universit\'e catholique de Louvain, B-1348 Louvain-la-Neuve, Belgium}
\email{barry.demba@gmail.com}
\begin{document}
\begin{abstract}
Let $A$ be a central simple algebra over a field $F$. Let $k_1,\ldots, k_r$ be cyclic extensions of $F$  such that $k_1\otimes_F\cdots \otimes_F k_r$ is a field. We investigate conditions under which $A$ is a tensor product of symbol algebras where each $k_i$ is in a symbol $F$-algebra factor of the  same degree as $k_i$. As an application, we give an example of an indecomposable algebra of degree $8$ and exponent $2$ over a  field of $2$-cohomological dimension $4$.



\end{abstract}
\maketitle
\section{Introduction}
Let $F$ be a field and let $k$ be a Galois extension of $F$ of degree $n$ with cyclic group generated by $\sigma$. For $a\in F\m$, we let $(k,\sigma, a)$ denote the cyclic $F$-algebra generated over $k$ by a single element $y$ with defining relation $ycy\inv=\sigma(c)$ for $c\in k$ and $y^n=a$. If $F$ contains a primitive $n$-th root of unity $\zeta$, it follows from Kummer theory that one may write $k$ in the form $k=F(\sqrt[n]{b})$ for some $b\in F\m$. The algebra $(k, \sigma, a)$ is then isomorphic to the \emph{symbol} algebra $(a,b)_n$ over $F$, that is, a central simple $F$-algebra generated  by two elements  $i$ and $j$ satisfying   $i^n=a$,  $j^n = b$ and $ij=\zeta ji$ (see for instance \cite[\S 15. 4]{Pie82}). In the case $n=2, \zeta=-1$, one gets a quaternion algebra over $F$ that will be denoted $(a, b)$.

A central division algebra decomposes into a tensor product of symbol algebras (of degree $2$ in \cite{ART79} and of degree an arbitrary prime $p$ in \cite{Row82}) if and only if it contains a set whose elements satisfy some commuting properties: \emph{$q$-generating set} in \cite{ART79}, \emph{$p$-central set} in \cite{Row82}, and a set of representatives  of an \emph{armature} in \cite{Tig82}. Our approach  is based on these notions. 

The main goal of this paper is to further investigate the decomposability of central simple algebras; the study of power central-elements is a constant tool. Let $A$ be a central simple algebra over $F$ and let $k_1,\ldots, k_r$ be cyclic extensions of $F$, of respective degree $n_1,\ldots, n_r$, contained in $A$ such that $k_1\otimes_F \cdots \otimes_F k_r$ is a field. It is natural to ask: \emph{when does there exist a decomposition of $A$ into a tensor product of symbol algebras in which each  $k_i$ is in a symbol $F$-subalgebra factor of $A$ of degree $n_i$? }  Starting from $A$, we construct a division algebra $E$ whose center is the iterated Laurent series with $r$ indeterminates over $F$ and show that $A$ admits such a decomposition if and only if $E$ is a tensor product of symbol algebras (see Theorem \ref{thm3.1} and Corollary \ref{thm3.2} for details). Note that  Corollary \ref{thm3.2} is very close to a result of Tignol \cite[Prop. 2.10]{Tig87}. In contrast with  Tignol's result (which is stated in terms of Brauer equivalence and only for prime exponent), Corollary \ref{thm3.2} is stated in terms of isomorphism classes and is valid for any exponent. Moreover, our approach  is completely different. We will give an example, pointed out by Merkurjev, of a division algebra which is a tensor product of three quaternion algebras, and containing a quadratic field extension which is in no quaternion subalgebra (Corollary \ref{cor3.1}). Using valuation theory, we give a general  method for constructing  tensor products of quaternion algebras containing a quadratic field extension which is in no quaternion subalgebra. As an application, let $A$ be a central simple algebra of degree $8$ and exponent $2$ over $F$, and containing a quadratic field extension which is in no quaternion subalgebra.   We use Corollary \ref{thm3.2} to associate with $A$ an example of an indecomposable algebra of degree $8$ and exponent $2$ over a field of rational functions in one variable over a field of $2$-cohomological dimension $3$ (see Theorem \ref{application}).  This latter field is obtained by an inductive process pioneered by Merkurjev \cite{Mer92}.

We next recall some results related to our main question: let $A$ be a $2$-power dimensional central simple algebra over $F$ and  let $F(\sqrt{d_1}, \sqrt{d_2})\subset A$ be a biquadratic field extension of $F$. If $A$ is a biquaternion algebra,  it follows from a result of Albert that $A\simeq (d_1, d_1')\otimes_F (d_2, d_2')$ where $d_1', d_2'\in F\m$ (see for instance \cite{Rac74}). As observed above, this is not true anymore in higher degree.  More generally, if $A$ is  of degree $8$ and exponent $2$ and $F(\sqrt{d})\subset A$ is a quadratic field extension, there exists  a cohomological criterion associated with the centralizer of $F(\sqrt{d})$ in $A$ which determines whether $F(\sqrt{d})$ lies in a quaternion $F$-subalgebra   of $A$ (see \cite[Prop. 4.4]{Bar}).  In the particular case where the $2$-cohomological dimension of $F$ is $2$ and $A$ is a division algebra of exponent $2$ over $F$, the situation is more favorable: it is shown in \cite[Thm. 3.3]{Bar} that  there  exits a decomposition of $A$ into a tensor product of quaternion $F$-algebras in which each $F(\sqrt{d_i})$ (for $i=1,2$)   is in a quaternion $F$-subalgebra.  
 
An outline of this article is the following: in Section \ref{sect2} we collect from \cite{Tig82} and \cite{TW87} some results on armatures of algebras  that will be used in the proofs of the main results. Section \ref{sect3} is devoted to the statements and the proofs of the main results. The particular case of exponent $2$ is analyzed (in more details) in Section \ref{sect4}.
 
All  algebras considered in this paper  are associative and finite-dimensional over their center. A central simple algebra $A$ over a field $F$ is \emph{decomposable} if $A \simeq A_1 \otimes_F A_2$ for two central simple $F$-algebras $A_1$ and $A_2$ both non isomorphic to $F$; otherwise A is called \emph{indecomposable}.  

Throughout this article, we shall use freely the standard terminology and notation from the theory of finite-dimensional algebras and the theory of valuations on division algebras. For these, as well as background information, we refer the reader to Pierce's book \cite{Pie82}.
\section{Armatures of algebras}\label{sect2}
Armatures in central simple algebras are a major tool for the next section. The goal of this section is to recall the notion of an armature and gather some preliminary results that will be used in the sequel.

We write $|H|$ for the cardinality of a set $H$. Let $A$ be a central simple $F$-algebra. For $a\in A^{\times}/F^{\times}$, we fix an element  $x_a$ of $A$ whose image in $A^{\times}/F^{\times}$ is $a$, that is, $a=x_aF^{\times}$. For a finite subgroup  $\cA$ of $A^{\times}/F^{\times}$, 
\[
F[\cA]= \Big\{\sum_{a\in\cA}c_ax_a\,|\, c_a\in F \Big\}
\]
denotes the $F$-subspace of $A$ generated by $\{x_a\,|\, a\in \cA\}$. Note that this subspace is independent of the choice of representatives $x_a$ for $a\in\cA$. Since $\cA$ is a group, $F[\cA]$ is  the subalgebra of $A$ generated by $\{x_a\,|\, a\in \cA\}$. As was observed in \cite{Tig82}, if $\cA$ is a finite abelian subgroup of $A^{\times}/F^{\times}$ there is an associated pairing $\langle\,,\,\rangle$ on $\cA\times\cA$ defined by 
\[
\langle a,b\rangle=x_ax_bx^{-1}_ax_b^{-1}.
\]
This definition is independent of the choice of representatives $x_a, x_b$ for $a,b$ and $\langle a,b\rangle$ belongs to $F^{\times}$ as $\cA$ is abelian.  Hence, $\langle a,b\rangle$ is central in $A$, and it follows that the pairing $\langle\,,\,\rangle$ is bimultiplicative. It is also alternating, obviously. Thus, as $\cA$ is finite, the image of $\langle\,,\,\rangle$ is a finite subset of $\mu(F)$ (where $\mu(F)$ denotes the group of roots of unity of $F$). For any subgroup $\cH$ of $\cA$ let 
\[
\cH^{\perp}=\{a\in \cA\,|\,\, \langle a,h\rangle =1 \,\text{ for all } h\in \cH\},
\]
a subgroup of $\cA$. The subgroup $\cH^{\perp}$ is called the \emph{orthogonal} of $\cH$ with respect to $\langle\,,\,\rangle$. The \emph{radical}\index{radical} of $\cA$, rad$(\cA)$, is defined to be $\cA^{\perp}$. The pairing $\langle\,,\,\rangle$ is called \emph{nondegenerate}\index{nondegenerate} on $\cA$ if $\text{rad}(\cA)=\{1_{\cA}\}$. 

For $g\in \cA$, we denote by $(g)$ the cyclic subgroup of $\cA$ generated by $g$ . The set $\{g_1,\ldots,g_r\}$ is called a base of $\cA$ if $\cA$ is the internal direct product 
\[ 
\cA=(g_1)\times\cdots\times (g_r).
\]
If $\langle\,,\,\rangle$ is nondegenerate then $\cA$ has a \emph{symplectic base}\index{symplectic base} with respect to $\langle\,,\,\rangle$, i.e, a base $\{g_1,h_1,\ldots, g_n,h_n\}$ such that for all $i,j$
\[
\langle g_i,h_i\rangle= c_i,\text{ where } \ord(g_i)=\ord(h_i)=\ord(c_i)
\]
\[
\langle g_i,g_j\rangle=\langle h_i,h_j\rangle=1 \text{ and, if } i\ne j,\,\, \langle g_i,h_j\rangle=1
\]
(see \cite[1.8]{Tig82}).
\begin{dfn}
For any finite-dimensional $F$-algebra $A$, a subgroup $\cA$ of $A^{\times}/F^{\times}$ is an \emph{armature}\index{armature} of $A$ if $\cA$ is abelian, $|\cA|=\dim_F A$, and $F[\cA]=A$.
\end{dfn}
If $\cA=\{a_1,\ldots, a_n\}$ is an armature of $A$, the above definition shows that the set $\{x_{a_1},\ldots,x_{a_n}\}$ is an $F$-base of $A$. The notion of an armature was introduced by Tignol in \cite{Tig82} for division algebras. The definition given here, slightly different from that given in \cite{Tig82}, comes from \cite{TW87}. This definition allows armatures in algebras other than division algebras. The following examples will be used repeatedly in the next section. 
\begin{examples}\label{ex3.1}
(a) (\cite{Tig82}) Let $A=A_1\otimes_F\cdots\otimes_F A_r$ be a tensor product of symbol $F$-algebras where $A_k$ is a symbol subalgebra of degree $n_k$. Suppose $F$ contains a primitive $n_k$-th root $\zeta_k$ of unity for $k=1,\ldots,r$. So, $A_k$ is isomorphic to a symbol algebra $(a_k,b_k)_{\zeta_k}$ of degree $n_k$ for some $a_k, b_k\in F^\times$. For each $k$, let $i_k,j_k$ be a symbol generator of $(a_k,b_k)_{\zeta_k}$. The image $\cA$ in $A^{\times}/F^{\times}$ of the set 
\[
\{i_1^{\alpha_1}j_1^{\beta_1}\ldots i_r^{\alpha_r}j_r^{\beta_r}\,|\,\, 0\le \alpha_k,\beta_k\le n_k-1\}
\]
is an armature of $A$ isomorphic to $(\Z/n_1\Z)^2\times\cdots\times (\Z/n_r\Z)^2$. We observe that for all $1\ne a\in\cA$ there exists $b\in \cA$ such that $\langle a,b\rangle\ne 1$; that is the pairing $\langle\,,\,\rangle$ is nondegenerate on $\cA$. Furthermore, $\{i_1F^\times, j_1F^\times,\ldots, i_rF^\times, j_rF^\times\}$ is a symplectic base of $\cA$. 

\vspace{0.2cm}
(b) Let $M$ be a finite abelian extension of a field $F$ and let $G$ be the Galois group of $M$ over $F$. Let $\ell$ be the exponent of $G$. If $F$ contains an $\ell$-th primitive root of unity, the extension $M/F$ is called a \emph{Kummer extension}\index{Kummer extension}. Let
\[
S=\{x\in M^{\times}\,\,|\,\, x^\ell\in F^{\times}\} \quad \text{and} \quad \text{\sf Kum}(M/F)= S/F^{\times}.
\]
It follows from Kummer theory (see for instance \cite[p.119-123]{Jac64}) that $\text{\sf Kum}(M/F)$ is a subgroup of $M^{\times}/F^{\times}$ and is dual to $G$ by the nondegenerate Kummer pairing $G\times \text{\sf Kum}(M/F)\to \mu(F) $ given by $(\sigma, b)=\sigma(x_b)x_b\inv$, for $\sigma\in G$ and $b\in \text{\sf Kum}(M/F)$.  Whence, $\text{\sf Kum}(M/F)$  is isomorphic (not canonically in general) to $G$. As observed in \cite[Ex. 2.4]{TW87}, the subgroup $\text{\sf Kum}(M/F)$ is the only armature of $M$ with exponent dividing $\ell$.
\end{examples}
Let $\cA$ be an armature of a central simple $F$-algebra $A$ and let $\{a_1,b_1,\ldots, a_n,b_n\}$ be a symplectic base of $\cA$ with respect to $\langle\,,\,\rangle$. We shall denote by $ F[(a_k)\times(b_k)]$ the subalgebra of $A$ generated by the representatives $x_{a_k}$ and $x_{b_k}$ of $a_k$ and $b_k$.  It is clear that $F[(a_k)\times(b_k)]$  is a symbol subalgebra of $A$ of degree $n_k=\ord(a_k)=\ord(b_k)$ and generated by  $x_{a_k}$ and $x_{b_k}$. It is shown in \cite[Lemma 2.5]{TW87} that $A\simeq F[(a_1)\times(b_1)]\otimes_F\cdots\otimes_F F[(a_n)\times(b_n)]$.

\vspace{0.2cm}
Actually, the notion of an armature is a generalization and refinement of the notion of a quaternion generating set ($q$-generating set) introduced in  \cite{ART79}. Indeed, a central simple algebra $A$ over $F$  has an armature if and only if $A$ is isomorphic to a tensor product of symbol algebras over $F$ (see \cite[Prop. 2.7]{TW87}).

\vspace{0.2cm}
Note that if the exponent of $\cA$ is a prime $p$, we may consider $\cA$ as a vector space over the field with $p$ elements $\mathbb F_p$. Identifying the group of $p$-th roots of unity with $\mathbb F_p$ by a choice of a primitive $p$-th root of unity, we may suppose the pairing has  values in $\mathbb F_p$. So, two elements $a,b\in\cA$ are orthogonal if and only if $\langle a,b\rangle=0$.  We need the following proposition:
\begin{proposition}\label{prop3.1}
Let $V$ be a vector space over $\mathbb F_p$ of dimension $2n$ and let $\langle\,,\,\rangle$ be a nondegenerate alternating pairing on $V$. Let $\{e_1,\ldots,e_r\}$ be a base of a totally isotropic subspace of $V$ with respect to $\langle\,,\,\rangle$. There are $f_1, \ldots,f_r,e_{r+1},f_{r+1}, \ldots,e_n, f_n$ in $V$ such that $\{e_1,f_1,\ldots,e_n,f_n\}$ is a symplectic base of $V$.
\end{proposition}
\begin{proof}
We argue by induction on the dimension of the totally isotropic subspace spanned by $e_1,\ldots,e_r$. If $r=1$, since the pairing is nondegenerate, there is $f_1\in V$ such that $\langle e_1,f_1\rangle\ne 0$ . Denote by $U=\text{span}(e_1,f_1)$ the subspace spanned by $e_1,f_1$. We have $V=U\perp U^{\perp}$ since the restriction of $\langle\,,\,\rangle$ to $U$ is nondegenerate. We  take for $\{e_2, f_2,\ldots, e_n, f_n\}$ a symplectic base of $U^{\perp}$.

Assume the statement for a totally isotropic subspace of dimension $r-1$. Let $W=\text{span}(e_2,\ldots,e_r)$; we have $W\subset W^{\perp}$. First, we  find $f_1\in W^{\perp}$ such that $\langle e_1,f_1\rangle\ne 0$. For this, consider the induced pairing, also denoted by $\langle\,,\,\rangle$, on $W^{\perp}/W$ defined by $\langle x+W,y+W\rangle=\langle x,y\rangle$ for $x,y\in W^{\perp}$. It is well-defined, and nondegenerate since $(W^{\perp})^{\perp}=W$. The element $e_1+W$ being non-zero in $W^{\perp}/W$, there is $f_1+W\in W^{\perp}/W$ such that $0\ne\langle e_1,f_1\rangle =\langle e_1+W,f_1+W\rangle$.  Letting $U=\text{span}(e_1,f_1)$, we have $V=U\perp U^{\perp}$ and $e_2,\ldots,e_r\in U^{\perp}$. Induction yields $f_2,\ldots, f_r, e_{r+1}, f_{r+1},\ldots, e_n, f_n\in U^{\perp}$ such that $\{e_2, f_2,\ldots, e_n, f_n\}$ is a symplectic base of $U^{\perp}$. Then $\{e_1, f_1,\ldots, e_n, f_n\}$ is a symplectic base of $V$.
\end{proof}

\section{Decomposability}\label{decomposability}\label{sect3}
Let $A$ be a central simple algebra over $F$ and let $t_1,\ldots,t_r$ be independent indeterminates over $F$. For $i=1,\ldots,r$, let $k_i$ be a cyclic extension of $F$ of degree $n_i$  contained in  $A$. We assume that $M=k_1\otimes_F\cdots \otimes_F k_r$ is a field and denote by $G$ the Galois group of $M$ over $F$. So $[M:F]=n_1\ldots n_r$ and $G=\langle \sigma_1\rangle\times\cdots \times\langle \sigma_r\rangle$, where $\langle \sigma_i\rangle$ is the Galois group of $k_i$ over $F$, and the order of $G$ is $n_1\ldots n_r$. Every element $\sigma\in G$ can be expressed as $\sigma=\sigma_1^{m_1}\ldots \sigma_r^{m_r} $ ($0\le m_i< n_i$). We shall denote by $C=C_AM$ the centralizer of $M$ in $A$. Let $t_1,\ldots, t_r$ be independent indeterminates over $F$. Consider the fields
\[
L'=F(t_1, \ldots, t_r) \quad \text{ and }\quad L= F((t_1))\ldots((t_r))
\]
and the following central simple algebras over $L'$ and $L$ respectively
\[
N'=(k_1\otimes_FL', \sigma_1\otimes\text{id},t_1)\otimes_{L'}\cdots\otimes_{L'}(k_r\otimes_FL', \sigma_r\otimes\text{id},t_r)
\]
and
\[
N=(k_1\otimes_FL, \sigma_1\otimes\text{id},t_1)\otimes_L\cdots\otimes_L(k_r\otimes_FL, \sigma_r\otimes\text{id},t_r).
\]
 We let
\[
R'=A\otimes_F N'\quad \text{ and }\quad R=A\otimes_F N.
\]
In this section our goal is to prove the following results:
\begin{theorem}\label{thm3.1}
Let $\cM$ be a finite group with a nondegenerate alternating pairing $\cM\times \cM \to \mu(F)$. Suppose $C$ is a division subalgebra of $A$, and $F$ contains a primitive $\exp(\cM)$-th root of unity. Then, the following are equivalent:
\begin{enumerate}
\item[(i)] The division algebra Brauer equivalent to $R'$ (respectively $R$)  has an armature isomorphic to $\cM$.
\item[(ii)] The algebra $A$ has an armature  isomorphic to $\cM$ and containing $\text{\sf Kum}(M/F)$ as a totally isotropic subgroup.
\end{enumerate} 
\end{theorem}

\vspace{0.3cm}
In the particular case where $A$ is of degree $p^n$ and exponent $p$ (for a prime number $p$) and each $k_i$ is a cyclic extension of $F$ of degree $p$, we have:
\begin{corollary}\label{thm3.2}
Assume that $A$ is of degree $p^n$ and exponent $p$. Suppose $C$ is a division subalgebra of $A$,  $F$ contains a primitive $p$-th root of unity, and $n_i=p$ for $i=1,\ldots, r$. Then, the following are equivalent:
\begin{enumerate}
\item[(i)] The division algebra Brauer equivalent to $R'$ (respectively $R$) is decomposable into a tensor product of symbol algebras of degree $p$.
\item[(ii)] The algebra   $A$ decomposes as 
\[
A\simeq (k_1,\sigma_1, \delta_1)\otimes_F\cdots\otimes_F(k_r,\sigma_r, \delta_r)\otimes_F A_{r+1}\otimes_F\cdots \otimes_F A_n
\]
for some $\delta_1,\ldots,\delta_r\in F^\times$ and some symbol algebras $A_{r+1},\ldots, A_n$ of degree $p$.  
\end{enumerate}
Moreover, if these conditions are satisfied then the division algebras Brauer equivalent to $R'$ and $R$ decompose respectively as
\[
(k_1\otimes_FL', \sigma_1\otimes\text{id},\delta_1t_1)\otimes_{L'}\cdots\otimes_{L'}(k_r\otimes_FL', \sigma_r\otimes\text{id},\delta_rt_r)\otimes_F A_{r+1}\otimes_F\cdots \otimes_F A_n
\]
and
\[
(k_1\otimes_FL, \sigma_1\otimes\text{id},\delta_1t_1)\otimes_{L}\cdots\otimes_{L}(k_r\otimes_FL, \sigma_r\otimes\text{id},\delta_rt_r)\otimes_F A_{r+1}\otimes_F\cdots \otimes_F A_n.
\] 

\end{corollary}

\vspace{0.3cm}
As opposed to  part  (i),  it is not enough in  part (ii)  of Theorem \ref{thm3.1}  to assume simply that $A$ has an armature to get an armature in the division algebra Brauer equivalent to $R'$ or $R$.  The  following example shows that the existence of an armature and the existence of an armature containing $\text{\sf Kum}(M/F)$ are different.
\begin{example}
Denote by $\Q_2$ the field of $2$-adic numbers, and  let $A$ be a division algebra of degree $4$ and exponent $4$ over $F=\Q_2(\sqrt{-1})$. Such an algebra is a symbol (see \cite[Th., p. 338]{Pie82}).  Note that $M=F(\sqrt{2},\sqrt{5})$ is a field since the set $\{-1, 2, 5\}$ forms a $\Z/2\Z$-basis of $\Q_2^\times/\Q_2^{\times 2}$ (see for instance \cite[Lemma 2.24, p. 163]{Lam05}). It also follows by \cite[Prop., p. 339]{Pie82} that $A\otimes M$ is split. Hence, since $[M:F]$ divides $\deg(A)$, we deduce that $M\subset A$.  Moreover, it is clear that the centralizer of $M$ in $A$ is M. The algebra $A$ being a symbol, it has an armature $\cA$ isomorphic to $(\Z/4\Z)^2$ (see Example \ref{ex3.1}). Assume that $\cA$ contains $\text{\sf Kum}(M/F)\simeq (\Z/2\Z)^2$. Since $\cA^2=\{a^2\,|\,\, a\in \cA\}=\text{\sf Kum}(M/F)$, the algebra $A$ is a symbol of the form $A=(c,d)_4$ with $c.F^{\times 2}$,  $d.F^{\times 2}\in \{2.F^{\times 2}, \,5.F^{\times 2},\, 2.5.F^{\times 2}\}$. It follows that $A\otimes A$ is either Brauer equivalent to the quaternion algebra $(2,5)$ or $(2, 2.5)$ or $(5, 2.5)$. Since $(2,5)\simeq (-1,-1)$ over $\Q_2$, the algebra $(2,5)$ is split over $F$; that is $A\otimes A$ is split. Therefore the exponent of $A$ must be $2$; impossible. Therefore $A$ has no armature containing $\text{\sf Kum}(M/F)$.

Now, let $E$ be the division algebra Brauer equivalent to $A\otimes (2,t_1)\otimes (5,t_2)$. As we will see soon (Lemma \ref{lm3.1} and Lemma \ref{lm3.2}), $\deg(E)=\exp(E)=4$. But $E$ has no armature, that is, $E$ is not a symbol. Indeed, suppose $E$ has an armature $\cB$. If $\cB\simeq (\Z/2\Z)^4$ then $E$ is a biquaternion algebra, so $\exp(E)=2$; contradiction. Therefore $\cB$ is  isomorphic to  $(\Z/4\Z)^2$.  It follows then by Theorem \ref{thm3.1}  that $A$ has an armature containing the amature $\text{\sf Kum}(M/F)$ of $M$. This is impossible as we showed above; therefore $E$ has no armature. 
\end{example}
\subsection{Brauer classes of $R'$ and $R$}
For the proof of the results above, we need an explicit description of the division algebras Brauer equivalent to $R'$ and $R$. First, we fix some notation: recall that we denoted by $G=\langle\sigma_1\rangle\times\cdots\times \langle\sigma_r\rangle$ the Galois group of $M$ over $F$. By the Skolem-Noether Theorem, for each $\sigma_i$ there exists $z_i\in A^\times$ such that $\sigma_i(b)=z_ibz_i\inv$ for all $b\in M$. Notice that  $z_i^{n_i}=:c_i\in C$ and $z_iz_jz_i\inv z_j\inv=:u_{ij}\in C$ for all $i,j$. For all $\sigma=\sigma_1^{m_1}\ldots\sigma_r^{m_r}\in G$, we set $z_{\sigma}=z_1^{m_1}\ldots z_r^{m_r}$ ($0\le m_i<n_i$). Setting $c(\sigma,\tau)=z_\sigma z_\tau (z_{\sigma\tau})\inv$, a simple observation shows for all $\sigma,\tau$
\begin{equation}\label{eq3.1}
c(\sigma,\tau)\in C \quad \text{ and }\quad z_\sigma b=\sigma(b)z_\sigma \text{ for all } b\in M.
\end{equation}
In fact, $c(\sigma,\tau)$ can be calculated from the elements $u_{ij}$ and $c_i$. 

Let $y_i$ be a generator of $(k_i\otimes L', \sigma_i\otimes 1, t_i)$, that is an element satisfying the relations $y_i(d\otimes l)y_i^{\inv}  = \sigma_i(d)\otimes l$ for all $d\in  k_i$  and $l\in L'$ and $y_i^{n_i}=t_i$. For $\sigma=\sigma_1^{m_1}\ldots \sigma_r^{m_r}\in G$ with $0\le m_i< n_i$, we set $y_\sigma=y_1^{m_1}\ldots  y_r^{m_r}$.
The algebras $N'$ and $N$ being crossed products, we may write
\[
N'=\bigoplus_{\sigma\in G}(M\otimes L')y_\sigma \quad \text{ and } \quad N=\bigoplus_{\sigma\in G}(M\otimes L)y_\sigma
\] 
and the $y_\sigma$ satisfy
\begin{equation}\label{eq3.2}
 y_\sigma y_\tau(y_{\sigma\tau})\inv \in M\otimes L \text{ and } y_\sigma(b\otimes l)=(\sigma(b)\otimes l)y_\sigma 
\end{equation}
for all  $b\in M$  and  $l\in L$.  Set $f(\sigma,\tau)= y_\sigma y_\tau(y_{\sigma\tau})\inv$. The elements  $f(\sigma,\tau)$ are in fact in  $L'$.  Indeed, let $y_\sigma= y_1^{\alpha_1}\ldots y_r^{\alpha_r}$ and $y_\tau= y_1^{\beta_1}\ldots y_r^{\beta_r}$ with $0\le \alpha_i, \beta_i < n_i$ . We get 
\begin{equation}\label{eq3.2.0}
f(\sigma,\tau)=y_\sigma y_\tau(y_{\sigma\tau})\inv=t_1^{\varepsilon_1}\ldots t_r^{\varepsilon_r}
\end{equation}
where $\varepsilon_i=0$ if $\alpha_i+\beta_i< \ord(\sigma_i)= n_i$ and $\varepsilon_i=1$ if $\alpha_i+\beta_i\geq \ord(\sigma_i)$.  Note that $f(\sigma,\tau)=f(\tau,\sigma)$ for all $\sigma,\tau\in G$.

Now, let $e$ be the \emph{separability idempotent} of $M$, that is, the idempotent $e\in M\otimes_F M$ determined uniquely by the conditions that
\begin{equation}\label{eq3.3}
e\cdot(x\otimes 1)=e\cdot(1\otimes x) \quad \text{for all } x\in M
\end{equation}
and the multiplication map $M\otimes_F M\,\to\, M$ carries $e$ to $1$ (see for instance \cite[\S 14. 3]{Pie82}). For $\sigma\in G$, let
\[
e_\sigma=(\text{id}\otimes \sigma)(e)\in M\otimes_F M.
\]
The elements $(e_\sigma)_{\sigma\in G}$ form a family of orthogonal primitive idempotents of $ M\otimes_F M$ (see \cite[\S 14. 3]{Pie82} ) and it follows, by applying $\text{id}\otimes\sigma$ to each side of (\ref{eq3.3}), that
\begin{equation}\label{eq3.4}
e_\sigma\cdot(x\otimes 1)=e_\sigma\cdot(1\otimes \sigma(x)) \quad \text{for } x\in M.
\end{equation} 
We need the following lemma:
\begin{lemma}\label{lm3.1}
Notations are as above.
\begin{enumerate}
\item The elements  $z_{\sigma}\otimes y_\sigma$ are subject to the following rules
\begin{itemize}
\item[(i)] For all $\sigma,\tau\in G$, there exists  $u\in C_L$ such that 
\[
 (z_\sigma\otimes y_\sigma)(z_\tau\otimes y_\tau)=(u\otimes 1)(z_{\sigma\tau}\otimes y_{\sigma\tau}).
\]
\item[(ii)]  $(z_\sigma\otimes y_\sigma)(c\otimes 1)=((z_{\sigma}cz^{-1}_{\sigma})\otimes 1)(z_\sigma\otimes y_\sigma)$ for all $c\in C$. Moreover, $z_{\sigma}cz^{-1}_{\sigma}\in C$ for all $\sigma\in G$. 
\end{itemize}
\item The sums
\[
E'=\sum_{\sigma\in G}C_{L'}z_\sigma\otimes y_\sigma \quad\text{and} \quad E=\sum_{\sigma\in G}C_{L}z_\sigma\otimes y_\sigma
\] 
are direct and are central simple subalgebras of $R'$ and $R$ respectively. Moreover, the algebras $R'$ and $R$ are Brauer equivalent to $E'$ and $E$ respectively, and $\deg E'=\deg E=\deg A$.
\end{enumerate}
\end{lemma}
\begin{proof}
(1) The statement of (i) follows from relations (\ref{eq3.1}),  (\ref{eq3.2}) and the fact that $f(\sigma,\tau)\in L'$. Since the elements  $1\otimes y_\sigma$ centralize $C_{L}\otimes 1$, the statement of (ii) is clear.

(2) Suppose  $\sum_{\sigma\in G}c_\sigma z_\sigma\otimes y_\sigma=0$ with $c_\sigma\in C_{L'}$ (or $C_{L}$).  Pick such a sum with a minimal number of non-zero terms. There are at least two non-zero elements, say $c_\rho z_\rho\otimes y_\rho$, $c_\tau z_\tau \otimes y_\tau$, in the sum. Let $b\in M$ be such that $\rho(b)\ne b$ and $\tau(b)=b$. One has
\[
(b\otimes 1)\Big(\sum_{\sigma\in G}c_\sigma z_\sigma\otimes y_\sigma\Big)(b\otimes 1)\inv-\sum_{\sigma\in G}c_\sigma z_\sigma\otimes y_\sigma=0
\]
and the number of non-zero terms is nontrivial and strictly smaller; contradiction. Whence we have the direct sums $E'=\bigoplus_{\sigma\in G}C_{L'}z_\sigma\otimes y_\sigma$ and $E=\bigoplus_{\sigma\in G}C_{L}z_\sigma\otimes y_\sigma$. It is clear that $E'\subset R'$ and $E\subset R$. On the other hand, the computation rules of the part (1) show that $E'$ and $E$ are generalized crossed products (see \cite[Th. 11.11]{Alb61} or \cite[\S 1.4]{Jac96}). Hence, the same arguments as for the usual crossed products show that $E'$ and $E$ are central simple algebras over $L'$ and $L$ respectively.  Moreover, since $\dim C = \frac{1}{n_1\ldots n_r}\dim A$, we have $\deg E'=\deg E=\deg A$.

Now, it remains to show that $R'$ and $R$ are respectively Brauer equivalent to $E'$ and $E$. For this we work over $L$; the same arguments apply over $L'$.  For $z_\sigma\otimes y_\sigma$ as above, consider the inner automorphism
\[
\text{Int}(z_\sigma\otimes y_\sigma)\,:\, R\, \longrightarrow\, R.
\]
Notice that Int$(z_\sigma\otimes y_\sigma)(e_\tau)$ is in $M\otimes M$ and is a primitive idempotent  for all $\tau\in G$  (since $e_\tau$ is primitive). Moreover, for $x\in M$,
\begin{eqnarray*}
\text{Int}(z_\sigma\otimes y_\sigma)(e_\tau)\cdot(x\otimes 1) &=& (z_\sigma\otimes y_\sigma)e_\tau (z_\sigma\inv\otimes y_\sigma\inv)\cdot(x\otimes 1)\\
& = & (z_\sigma\otimes y_\sigma)e_\tau \cdot (\sigma\inv(x)\otimes 1)(z_\sigma\inv\otimes
 y_\sigma\inv)\\
 & = & (z_\sigma\otimes y_\sigma)e_\tau \cdot (1 \otimes \tau\sigma\inv(x) )(z_\sigma\inv\otimes
 y_\sigma\inv)\\
 & = & (1 \otimes \tau(x) )\cdot\text{Int}(z_\sigma\otimes y_\sigma)(e_\tau).
\end{eqnarray*}
Therefore $\text{Int}(z_\sigma\otimes y_\sigma)(e_\tau)=e_\tau$ by comparing with the definition of $e$ and the relation (\ref{eq3.4}). Hence, each $e_\tau$ centralizes $E$ in $R$. On the other hand, since the degree of the centralizer $C_RE$ of $E$ in $R$ is $n_1\ldots n_r$ and $(e_\tau)_{\tau\in G}\subset C_RE$, the algebra $C_RE$ is split. So $R$ is Brauer equivalent to $E$.
\end{proof}
\begin{lemma}\label{lm3.2}
Notations are as in Lemma \ref{lm3.1}. If  $C$ is a division algebra then there exists a unique valuation  on $E$ which extends the $(t_1,\ldots, t_r)$-adic valuation on $L$. Consequently  $E'$ and $E$ are division algebras.
\end{lemma}
\begin{proof}
The  $(t_1,\ldots, t_r)$-adic valuation on $L$ being Henselian, it extends to a unique valuation to each  division algebra over $L$ (see for instance \cite{Wad86}).  It follows that there is  a valuation  on $C_L$ extending the  $(t_1,\ldots, t_r)$-adic valuation on $L$.  More precisely this valuation is constructed as follows:   writing $C_L$ as
\[
C_L=\Bigg\{\sum_{i_1\in\Z}\cdots \sum_{i_r\in\Z}c_{i_1\ldots i_r}t_1^{i_1}\ldots t_r^{i_r} 
\left| \begin{array}{ll} c_{i_1\ldots i_r} \in C \text{ and }\\
\{ (i_1,\ldots, i_r)\, |\,  c_{i_1\ldots i_r}\ne 0 \text{ is well-ordered}\\
 \text{for the right-to-left lexicographic ordering}\} 
 \end{array}
\right.
 \Bigg\},
\]
computations show that the map $v: C_L\m \to \Z^r$ defined by   
\[
v\Big(\sum_{i_1\in\Z}\cdots \sum_{i_r\in\Z}c_{i_1\ldots i_r}t_1^{i_1}\ldots t_r^{i_r}\Big)= \min \{(i_1,\ldots, i_r) \,|\, c_{i_1\ldots i_r}\ne 0 \}  
\]
is a valuation. Clearly $v$ extends the  $(t_1,\ldots, t_r)$-adic valuation on $L$.  Recall that $N$ is a division  algebra over $L$ (see e.g. \cite[Ex. 3.6]{Wad02}). We also denote by $v$ the unique extension of the $(t_1,\ldots, t_r)$-valuation  to $N$. Since $y_i^{n_i}=t_i$, we have 
\[
v(y_i)=(0,\ldots,0,\frac{1}{n_i},0,\ldots, 0)\in \frac{1}{n_1}\Z\times\ldots \times\frac{1}{n_r}\Z .
\]
 Hence, for $y_\sigma = y_1^{m_i}\ldots y_r^{m_r}$ where $\sigma= \sigma_1^{m_1}\ldots \sigma_r^{m_r}$ (with $1\le m_i < n_i$), we have $v(y_\sigma)=(\frac{m_1}{n_1},\ldots, \frac{m_r}{n_r})$. It then follows that 
 \begin{equation}\label{lmeq1}
 v(y_\sigma)\not\equiv v(y_\tau)\bmod \Z^r  \quad \text{ if }\quad \sigma\ne \tau.
 \end{equation}

Now,  define a map $w\,:\, E^\times \to \frac{1}{n_1}\Z\times\ldots \times\frac{1}{n_r}\Z$ as follows: for any $\sigma\in G$ and any $c\in C_L^\times$, set 
\[
w(cz_\sigma\otimes y_\sigma)= v(c)+v(y_\sigma).
\]
For any $s\in E^\times$, $s$ has a unique representation $s=\sum_{\sigma\in G}c_\sigma z_\sigma\otimes y_\sigma$ with the $c_\sigma\in C_L$ and some $c_\sigma\ne 0$. Define
\[
w(s)= \min_{\sigma\in G}\{w(c_\sigma z_\sigma\otimes y_\sigma)\, |\, c_\sigma \ne 0\}.
\]
It follows by  (\ref{lmeq1}) that $w(c_\sigma z_\sigma\otimes y_\sigma)\ne w(c_\tau z_\tau\otimes y_\tau)$ for $\sigma\ne \tau$. Thus, there is a unique summand   $c_\iota z_\iota\otimes  y_\iota$ of $s$ such that $w(s)=w(c_\iota z_\iota\otimes y_\iota)$;  this $ c_\iota z_\iota\otimes  y_\iota$ is called the \emph{leading term} of $s$.

We are going to show that $w$ is a valuation on $E$. Let $s'=\sum_{\sigma\in G}d_\sigma z_\sigma\otimes y_\sigma\in E^\times$ with $d_\sigma\in C_L$ and $s+s'\ne 0$. Let $(c_\rho+ d_\rho)z_\rho\otimes  y_\rho$ be the leading term of $s+s'$. If $ c_\rho\ne 0$ and $ d_\rho\ne 0$, we have
\begin{eqnarray*}
w (( c_\rho + d_\rho)z_\rho\otimes  y_\rho) & = & v(c_\rho+d_\rho)+v(y_\rho)\\
& \geq & \min (v(c_\rho)+ v(y_\rho), v(d_\rho)+v(y_\rho))\\
& = & \min (w(c_\rho z_\rho\otimes y_\rho), w(d_\rho z_\rho\otimes y_\rho))\\
& \geq & \min (w(s), w(s')).
\end{eqnarray*}
Thus, $w(s+s')= w (( c_\rho + d_\rho)z_\rho\otimes  y_\rho))\geq \min (w(s), w(s'))$. This inequality still holds if $c_\rho =0$ or $d_\rho=0$.

By the usual argument,  we also check that
\begin{equation}\label{lmeq2}
\text{if } w(s)\ne w(s') \text{ then } w(s+s')=\min (w(s), w(s')).
\end{equation}
It remains to show that $w(ss')= w(s)+w(s')$. For $\sigma,\tau\in G$, recall that
\[
(z_\sigma\otimes y_\sigma)(z_\tau\otimes y_\tau)= c(\sigma,\tau)f(\sigma,\tau)(z_{\sigma\tau}\otimes y_{\sigma\tau})
\]
for some $c(\sigma,\tau)\in C^\times$ and  some $f(\sigma,\tau)\in L^\times$. It follows that, for $c_\sigma, d_\tau\in C^\times_L$,
\begin{eqnarray} \label{lmeq3}
w((c_\sigma z_\sigma\otimes y_\sigma)(d_\tau z_\tau\otimes y_\tau))  &=& w(c_\sigma(z_\sigma d_\tau z_\sigma\inv)(z_\sigma\otimes y_\sigma)(z_\tau\otimes y_\tau))  \quad \text{(Lemma \ref{lm3.1})} \nonumber\\
& = & w(c_\sigma(z_\sigma d_\tau z_\sigma\inv) c(\sigma,\tau)f(\sigma,\tau)(z_{\sigma\tau}\otimes y_{\sigma\tau}))\nonumber\\
& = & v(c_\sigma)+ v(d_\tau)+ v(f(\sigma,\tau)y_{\sigma\tau})\nonumber\\
& = & v(c_\sigma)+ v(y_\sigma)+ v(d_\tau)+ v(y_\tau) \nonumber\\
& = & v(c_\sigma z_\sigma\otimes y_\sigma) + v(d_\tau z_\tau\otimes y_\tau).
\end{eqnarray}
Hence, we have
\begin{eqnarray} \label{lmeq4}
w(ss') &=& w\Big(\sum_{\sigma,\tau}(c_\sigma z_\sigma\otimes y_\sigma)(d_\tau z_\tau\otimes y_\tau)\Big)\nonumber\\
& \geq & \min_{\sigma,\tau}\{w((c_\sigma z_\sigma\otimes y_\sigma)(d_\tau z_\tau\otimes y_\tau))\,|\, c_\sigma, d_\tau \ne 0\}\nonumber\\
& = & \min_{\sigma,\tau} \{w(c_\sigma z_\sigma\otimes y_\sigma)+ w(d_\tau z_\tau\otimes y_\tau)\,|\, c_\sigma, d_\tau \ne 0\}\nonumber\\
& \geq & w(s)+ w(s').
\end{eqnarray}
Let $c_\rho z_\rho\otimes y_\rho$ and $d_\iota z_\iota \otimes y_\iota$ be the leading terms of $s$ and $s'$ respectively. Set $s_1=s-c_\rho z_\rho\otimes y_\rho$ and $s'_1= s'-d_\iota z_\iota \otimes y_\iota$. So, 
\[
w(s)=w(c_\rho z_\rho\otimes y_\rho)< w(s_1) \quad \text{and}\quad  w(s')= w(d_\iota z_\iota \otimes y_\iota)< w(s'_1).
\]
Writing
\[
ss'=(c_\rho z_\rho\otimes y_\rho)(d_\iota z_\iota \otimes y_\iota)+ s_1(d_\iota z_\iota \otimes y_\iota)+(c_\rho z_\rho\otimes y_\rho)s_1' + s_1s_1',
\]
it follows by (\ref{lmeq3}) and (\ref{lmeq4}) that the first  summand in the right side of the above equality has valuation strictly smaller than the other three. Hence, by (\ref{lmeq2}) and (\ref{lmeq3}), one has $ss'\ne 0$ and
\[
w(ss')= w((c_\rho z_\rho\otimes y_\rho)(d_\iota z_\iota \otimes y_\iota))= w(s)+w(s').
\]
Therefore $w$ is a valuation on $E$. Since $E=E'\otimes_{L'}L$, the restriction of $w$ to $E'$ is also a valuation. The uniqueness of $w$ follows from its existence by \cite{Wad86}.
\end{proof}
Let  $D$ be a  division algebra with a valuation. The \emph{residue} division algebra of $D$ is denoted by $\overline{D}$.

We keep the notations  above.  Now, suppose $C$ is a division subalgebra of $A$ and denote by $\Gamma_E$ and $\Gamma_{L}$ the corresponding value groups of $E$ and $L$ respectively.  Furthermore, assume that $E$ has an armature $\cA$. The diagram
\[
\xymatrix   { 
    1 \ar[r] & L^\times \ar[r] \ar[d]^v  & E^\times \ar[r] \ar[d]^w & E^\times/L^\times \ar[r] \ar@{.>}[d] & 1  \\
    0 \ar[r] & \Gamma_L \ar[r] & \Gamma_E \ar[r] & \Gamma_E/\Gamma_L \ar[r] & 0}
\]
induces a homomorphism
\[
w'\,:\,\cA\subset  E^{\times}/L^\times\,\longrightarrow\, \Gamma_E/\Gamma_L.
\]
Put $\cA_0=\ker w'$  and let $a\in \cA_0$. The above diagram shows that there exists a representative $x_a$ of $a$ such that $w(x_a)=0$. Define 
\[
\bar\,\,\,:\,\cA_0\,\longrightarrow\,\overline{E}^\times/\overline L^\times=C^\times/F^\times
\]
by
\[
a=x_aL^{\times}\,\longmapsto\, \bar a=\bar x_aF^{\times}
\]
where $x_a$ is such that $w(x_a)=0$ and $\bar x_a$ is the residue of $x_a$. If $y_a$ is another representative of $a$ such that $w(y_a)=0$, there is $h\in L^\times$ with $w(h)=0$ such that $y_a=x_ah$. Hence $\bar y_a=\bar x_a\bar h$, that is $\bar y_a=\bar x_aF^\times$, so $\bar\,\,$ is well defined. We have:
\begin{proposition}\label{prop3.3}
Assume that $C$ is a division algebra and $\cA$ is an armature of $E$ as above. Then 
\begin{itemize}
\item[(1)] The map $\bar\,\, \,:\, \cA_0\,\longrightarrow\,C^\times/F^\times$ is an injective homomorphism.
\item[(2)] The image $\overline{\cA_0}$ of $\cA_0$ is an armature of $C$ over $F$.
\end{itemize}
\end{proposition}
\begin{proof}
(1) Let $a=x_aL^\times$ and $b=x_bL^\times$ be such that $w(x_a)=w(x_b)=0$.  Choosing $x_{ab}=x_ax_b$, we have $w(x_{ab})=0$. Therefore $\bar x_{ab}=\bar x_a\bar x_b$; this shows that $\bar\,\,$ is a homomorphism.

Let $c\in \ker\bar\,\,$  with $c\ne 1$ and let $x_c\in E\m$ be a representative of $c$ such that $w(x_c)=0$. Let $\bar x_c=\alpha \in F\m$; then $x_c=\alpha+ x_c'$ for some  $x_c'\in  E$ with $w(x_c')>0$. The pairing $\langle \,,\,\rangle$ being nondegenerate on $\cA$, there exists $d \in \cA$ such that $\langle d,c\rangle=\zeta$ for some $1 \ne\zeta\in \mu(F)$. Thus, 
\[
\overline{x_{d}x_cx_{d}^{-1}}=\zeta\bar x_c=\zeta\alpha.
\]
On the other hand, 
\[
x_{d}x_cx_{d}^{-1}=x_{d}\alpha x_{d}^{-1}+x_d x_c' x_{d}^{-1}=\alpha+x_{d} x_c' x_{d}^{-1}.
\]
Hence, we get
\[
\overline{x_{d}x_cx_{d}^{-1}}=\alpha \,\,\,\text{ since } \,\,\, w(x_{d}x_c' x_{d}^{-1})>0; \,\, \text{contradiction}.
\]
Therefore the map $\bar\,\,$ is injective.  Consequently,  we have $|\overline{\cA_0}|=|\cA_0|$.

(2)  We first show that $\vert \cA_0\vert\le \dim_F C$: since $C=\overline E$, it suffices to prove that $(\bar x_a)_{a\in\cA_0}$ are linearly independent over $F$. Let  $\sum_{a\in \cA_0}\lambda_a\bar x_a=0$,
with $\lambda_a\in F$, be a zero linear combination such that the set 
\[
S=\{a\in \cA_0\,\,|\,\,\lambda_a \ne 0\}
\]
is not empty and of least cardinality. For $s\in S$, let $x_s$ be a representative of $s$ in $E\m$ such that $w(x_s)=0$. We have 
\[
\bar x_s\Big(\sum_{a\in \cA_0}\lambda_a\bar x_a\Big)\bar x_s\inv=\sum_{a\in \cA_0}\langle a,s\rangle \lambda_a\bar x_a=0=\sum_{a\in \cA_0}\lambda_a\bar x_a.
\] 
Then the linear combination  $\sum_{a\in \cA_0}(1-\langle a,s\rangle) \lambda_a\bar x_a$ is zero and the number of non-zero coefficients is less than the cardinality of $S$ because $\langle s,s\rangle=1$. Therefore, $\langle a,s\rangle=1$ for all $a,s\in S$; this implies that $\bar x_s$ and $\bar x_{s'}$ commute for all $s, s'\in S$. It follows from \cite[Lemma 1.5]{Tig82} that the elements $\bar x_s$, for $s\in S$, are linearly independent; contradicting the fact that $S$ is not empty. Combining with the part (1), we get 
\[
\vert\overline{\cA_0}\vert=\vert\cA_0\vert\le\dim C=\frac{1}{n_1\ldots n_r}\vert\cA\vert.
\]
On the other hand, since $\vert w'(\cA)\vert=\frac{\vert\cA\vert}{\vert\cA_0\vert}$, we have $\vert w'(\cA)\vert\geq n_1\ldots n_r$. We already know that $\vert w'(\cA)\vert\le n_1\ldots n_r$ because $w'(\cA)\subset \Gamma_E/\Gamma_L$ and  $|\Gamma_E/\Gamma_L|=n_1\ldots n_r$.   It follows that $\vert w'(\cA)\vert= n_1\ldots n_r$ and $\vert\overline{\cA_0}\vert=\dim_F C$. Since we showed that $(\bar x_a)_{a\in \cA_0}$ are linearly independent,  the subgroup $\overline{\cA_0}$ is an armature of $C$ over $F$.
\end{proof}

\subsection{Proof of the main result}
Recall that the division algebras Brauer equivalent to $R'$ and $R$ are respectively $E'$ and $E$ by Lemma \ref{lm3.1}.
\begin{proof}[Proof of Theorem \ref{thm3.1}]
(i) $\Rightarrow$ (ii):  we give the proof for $E$; the proof for $E'$ follows because if $E'$ has an armature then $E = E' \otimes_{L'} L$  has an isomorphic armature. Assume that $E$ has  an armature $\cA$. Let $c\in \cA$  and let $c_\rho z_\rho\otimes y_\rho$ be the leading term of a representative  $x_c$ of $c$ in $E$. That is, $x_c=c_\rho z_\rho\otimes y_\rho+ x'_c$ with  $w(x_c)= w(c_\rho z_\rho\otimes y_\rho)$ and $w(x'_c)> w(x_c)$. Define the map 
\[
\nu\,:\,\cA\,\longrightarrow\, A^{\times}/F^{\times}
\]
by
\[
c=x_c.L^\times\,\longmapsto\,  c_\rho z_\rho.F^{\times}.
\]
If $y_c$ is another representative of $c$, we have $y_c=\ell x_c$ for some $\ell\in L\m$.  Note that the leading term of $y_c$ is the leading term of $\ell$ multiplied by $c_\rho z_\rho\otimes y_\rho$ (see the proof of Lemma \ref{lm3.2}); moreover, the leading term of $\ell$ lies in $F\m$.  One deduces that  $\nu(x_c.L\m)= \nu(y_c.L\m)$. So $\nu$ is well-defined. 

We show that $\nu(\cA)$ is an armature of $A$: first, we claim that $\nu$ is an injective homomorphism. Indeed,   let $a,b\in\cA$ with respective representatives $x_a$ and $x_b$. Let $c_\sigma z_\sigma\otimes y_\sigma$ and $d_\tau z_\tau\otimes y_\tau$ be  the leading terms of $x_a$ and $x_b$ respectively.  As showed in the proof of Lemma \ref{lm3.2}, the leading term of $x_ax_b$ is $(c_\sigma z_\sigma\otimes y_\sigma)(d_\tau z_\tau\otimes y_\tau)$. On the other hand, it follows by (\ref{eq3.1}), (\ref{eq3.2}), (\ref{eq3.2.0}) and Lemma \ref{lm3.1} that
\[
(c_\sigma z_\sigma\otimes y_\sigma)(d_\tau z_\tau\otimes y_\tau)=c_\sigma (z_\sigma d_\tau z_\sigma\inv) f(\sigma,\tau)c(\sigma,\tau)z_{\sigma\tau}\otimes y_{\sigma\tau}
\]
for some $f(\sigma,\tau)\in  L^\times$ and some  $c(\sigma,\tau)\in C^\times$, and $z_\sigma d_\tau z_\sigma\inv\in C$. 

Since $x_{ab}=x_ax_b\bmod L^\times$  (because $a,b\in\cA$), we may take $x_ax_b$ as a representative of $ab$, so
\[
x_{ab}= x_ax_b=  c_\sigma (z_\sigma d_\tau z_\sigma\inv) c(\sigma,\tau)z_{\sigma\tau}\otimes y_{\sigma\tau}+ x_{ab}'  \quad \text{ with }\quad  w(x_{ab}') >w(x_{ab}).
\]
Hence, it follows by the definition of $\nu$ that
\begin{eqnarray*}
\nu(ab) & = &c_\sigma (z_\sigma d_\tau z_\sigma\inv) c(\sigma,\tau)z_{\sigma\tau} \bmod F^\times\\
& = & (c_\sigma z_\sigma)(d_\tau z_\tau) \bmod F^\times\\
& = & \nu(a)\nu(b).
\end{eqnarray*}
Therefore, $\nu$ is a homomorphism.

For the injectivity, we start out  by proving that the pairing $\langle\,,\,\rangle$ is an isometry for $\nu$: by definition 
\[
\langle a,b \rangle=x_ax_bx_a^{-1}x_b^{-1}=\zeta \quad \text{ for some }\quad \zeta\in \mu(F).
\]
Hence
\[
x_ax_b= (c_\sigma z_\sigma\otimes y_\sigma)(d_\tau z_\tau\otimes y_\tau)+(x_ax_b)' =\zeta x_b x_a=\zeta (d_\tau z_\tau\otimes y_\tau)(c_\sigma z_\sigma\otimes y_\sigma)+\zeta(x_bx_a)',
\]
with $w((x_ax_b)')=w((x_bx_a)')> w(x_ax_b)$. Since $y_\sigma$ and $y_\tau$ commute,  it follows that
\[
(c_\sigma z_\sigma)(d_\tau z_\tau)=\zeta (d_\tau z_\tau)(c_\sigma z_\sigma).
\]
Therefore
\[
\langle \nu(a),\nu(b) \rangle=(c_\sigma z_\sigma)(d_\tau z_\tau)(c_\sigma z_\sigma)\inv(d_\tau z_\tau)\inv=\zeta.
\]
The pairing $\langle\, ,\,\rangle$ is then an isometry for $\nu$. Now, to see that $\nu$ is an injection, let $a\in\cA$ be such that $\nu(a)=1$. Since $\langle\, ,\,\rangle$ is an isometry for $\nu$,  for all $b\in\cA$, one has
\[
1= \langle \nu(a),\nu(b) \rangle= \langle a,b \rangle.
\]
We infer that $a=1$ because the pairing is nondegenerate on $\cA$. It follows that $\nu(\cA)$ is an abelian subgroup of $A\m/F\m$ isomorphic to $\cA$. Consequently, $\nu(\cA)$ is an armature of $A\m/F\m$  isomorphic to $\cM$ ($\simeq \cA$ by hypothesis) since  $\deg E=\deg A$ by Lemma \ref{lm3.1}. 

Now, we prove that $\text{\sf Kum}(M/F)\subset \nu(\cA)$. Since $\overline{\cA_0}$ is an armature of $C$ by Proposition \ref{prop3.3}, it follows by \cite[Lemma 2.5]{TW87} that rad$(\overline{\cA_0})$ is an armature of the center of $C$ which is $M$. The extension $M/F$ being a Kummer extension, Examples \ref{ex3.1} (b) indicates that  $\text{rad}(\overline{\cA_0})=\text{\sf Kum}(M/F)$. Therefore, we have $\text{\sf Kum}(M/F)\subset \nu(\cA) $ since $\nu$ is the identity on $\cA_0$.

(ii) $\Rightarrow$ (i):  let $\cB$ be an armature of $A$ isomorphic to $\cM$ and containing $\text{\sf Kum}(M/F)$ as a totally isotropic subgroup. We construct an isomorphic armature in $E'$. Note that if $E'$ has an armature, then $E$ also has an armature. For each $\sigma\in G$, set
\[
\cB_\sigma=\{a\in \cB\, | \, \langle a,b\rangle=\sigma(x_b)x_b\inv \text{ for all } b\in \text{\sf Kum}(M/F)\}.
\]
One easily checks that  $\cB_{\id}= \text{\sf Kum}(M/F)^\perp$ and for $a,c\in \cB_\sigma$, $ac\inv\in \cB_{\id}$. The sets $\cB_\sigma$ are  the cosets of $\cB_{\id}$ in $\cB$. So, we have the disjoint union  $ \cB=\bigsqcup_{\sigma\in G} \cB_\sigma$. On the other hand, for $a\in \cB_\sigma$ and $b\in \text{\sf Kum}(M/F)$, comparing the equality $\langle a,b\rangle=\sigma(x_b)x_b\inv$ and the definition $\langle a,b\rangle=x_ax_bx_a\inv x_b\inv$ we get $x_ax_b=\sigma(x_b)x_a\inv$; this  implies  $\cB_\sigma\subset C\m z_\sigma/F\m$. Now, let us denote 
\[
\cB_\sigma'=\{(x_{a_\sigma} \otimes y_\sigma).L^{' \times}\,\,|\,\, a_\sigma\in \cB_\sigma \} \quad \text{and}\quad \cB'=\bigsqcup_{\sigma\in G}\cB'_\sigma.
\] 
Note that $\cB_\sigma'\subset C_{L'}\m(z_\sigma\otimes y_\sigma)/L^{' \times}$ and we readily check that $\cB'$ is a subgroup of $E^{' \times}/L^{' \times}$.  We claim that $\cB'$ is an armature of  $E'$:  since $\deg A=\deg E'$, it follows by the definition of $\cB'$ that $|\cB'|=\dim E'$. Moreover, as in the part (2) of the proof of Proposition \ref{prop3.3}, one verifies that the representatives  of the elements of $\cB'$ in $E'$ are linearly independent over $L'$. It remains to show that $\cB'$ is commutative. Let $a_\sigma\in \cB_\sigma$ and $d_\tau\in \cB_\tau$ with $\sigma,\tau\in G$. By (\ref{eq3.2.0}),  $y_\sigma y_\tau = y_\tau y_\sigma$ because $f(\sigma,\tau)=f(\tau,\sigma)$. Furthermore, taking  $x_{a_\sigma}x_{d_\tau}$ as a representative of $a_\sigma d_\tau$ (since $\cB$ is an armature),  we have $  x_{a_\sigma}x_{d_\tau}= x_{a_\sigma d_\tau} =x_{d_\tau}x_{a_\sigma}$. The commutativity of $\cB'$ follows; and therefore $\cB'$ is an armature of $E'$.

Using the same arguments as above, we see that the map $\cB' \to \cB$ that carries $(x_{a_\sigma}\otimes y_\sigma).L^{' \times}$ to $x_{a_\sigma}.F\m$ is an isomorphism. Consequently, the armature $\cB'$ is also isomorphic to $\cM$. This concludes the proof.
\end{proof}
\begin{proof}[Proof of Corollary \ref{thm3.2}]
(i) $\Rightarrow$ (ii): as in the proof of Theorem \ref{thm3.1}, it is enough to give the proof for $E$. Assume that $E$ decomposes into a tensor product of symbol algebras of degree $p$. Recall that if $E$ decomposes into a tensor product of symbol algebras of degree $p$ then $E$ has an armature of exponent $p$ (see \cite[Prop. 2.7]{TW87}).  It follows by Theorem \ref{thm3.1} that $A$ decomposes into a tensor product of symbol algebras of degree $p$.  More precisely, if $\cA$ is an armature of $E$ of exponent $p$, we showed that $\nu(\cA)$ is an armature of $A$ isomorphic to $\cA$ and $\text{\sf Kum}(M/F)\subset \nu(\cA)$.  Now, let $x_1,\ldots, x_r\in A$ be such that  $k_i= F(x_i)$. The subgroup generated by $(x_i F^{\times})$ for $i=1,\ldots, r$ is $\text{\sf Kum}(M/F)$. 
The exponent of $\nu(\cA)$ being $p$, we may view $\nu(A)$ as a vector space over the field with $p$ elements. Since $M$ is a field, the elements $e_1:=x_1F^\times,\ldots, e_r:= x_rF^\times$ are linearly independent in $\nu(\cA)$.  On the other hand, $M$ being commutative, the subspace spanned by $e_1, \ldots, e_r$ is totally isotropic with respect to $\langle\,,\,\rangle$. It follows then by Proposition \ref{prop3.1} that there are $f_1,\ldots, f_r, e_{r+1}, f_{r+1},\ldots, e_n, f_n$ in $\nu(\cA)$ such that $\{e_1, f_1,\ldots, e_n, f_n \}$ is a symplectic base of $\nu(\cA)$. Expressing $f_i= y_iF^\times$ for $i=1,\ldots,r$ and $\nu(\cA)= \cA_1\times\ldots\times\cA_n$, where $\cA_i=(e_i)\times(f_i)$ for $i=1,\ldots,n$, we get
\[
A \simeq  (k_1,\sigma_1, \delta_1)\otimes_F\cdots\otimes_F(k_r,\sigma_r, \delta_r)\otimes_F F[\cA_{r+1}]\otimes_F\cdots\otimes_F F[\cA_n]
\]
with $\delta_i=y_i^p$ for $i=1,\ldots,r$. 

(ii) $\Rightarrow$ (i): assume that $A$ decomposes as
\[
A\simeq (k_1,\sigma_1, \delta_1)\otimes_F\cdots\otimes_F(k_r,\sigma_r, \delta_r)\otimes_F A_{r+1}\otimes_F\cdots \otimes_F A_n
\]
for some $\delta_1,\ldots,\delta_r\in F^\times$ and some symbols subalgebras $A_{r+1},\ldots, A_n$ of $A$. We give the proof for $R$. The same argument is valid for $R'$. We have 
\begin{multline*}
R= A\otimes_F (k_1\otimes_FL, \sigma_1\otimes\text{id},t_1)\otimes_L\cdots\otimes_L(k_r\otimes_FL, \sigma_r\otimes\text{id},t_r)\sim \\
(k_1\otimes_FL, \sigma_1\otimes\text{id},\delta_1t_1)\otimes_L\cdots\otimes_L(k_r\otimes_FL, \sigma_r\otimes\text{id},\delta_rt_r)\otimes_F A_{r+1}\otimes_F\cdots \otimes_F A_n
\end{multline*}
(see for instance \cite[\S 10]{Dra83}). Since this latter algebra  has the same degree as $A$ and $\deg(A)=\deg(E)$ by Lemma \ref{lm3.1}, it is isomorphic to $E$. The proof is complete.
\end{proof}

\section{Square-central elements}\label{square-central}\label{sect4}
Let $A$ be a central simple $F$-algebra of exponent $2$ and let $g\in A^{\times}-F$ be a square-central element. The purpose of this section is to investigate conditions for  $g$ to be in a quaternion subalgebra of $A$ and to give examples of tensor products of quaternion algebras containing a square-central element which is in no quaternion subalgebra.
\subsection{The algebra $A$ is not a division algebra}
Here we distinguish two cases, according to whether $g^2\in F^{\times 2}$ or $g^2\notin F^{\times 2}$. Actually, we will not need to mention in the  following proposition that $A$ is not a division algebra because this is encoded by the fact that $g\in A^{\times}-F^{\times}$ and $g^2\in F^{\times 2}$. Indeed, if $g^2=\lambda^2$ with $\lambda\in F^{\times}$ then $(g-\lambda)(g+\lambda)=0$; this means that $A$ is not division.
\begin{proposition}\label{prop1.1}
Let $A$ be a central simple $F$-algebra and let $g\in A^{\times}-F^{\times}$ be such that $g^2=\lambda^2$, $\lambda\in F^{\times}$. The element $g$ is in a quaternion subalgebra  of $A$ if and only if $\dim(g-\lambda)A=\dim(g+\lambda)A$.  If the characteristic of $F$ is $0$, this condition holds if and only if the reduced trace $\Trd_A(g)$ of $g$ is zero.
\end{proposition}
\begin{proof}
We can write $A\simeq \text{End}_D(V)$ where $D$ is a division algebra Brauer equivalent to $A$ and $V$ is some right $D$-vector space. Suppose there is a quaternion $F$-subalgebra $Q$ of $A$ such that $g\in Q$. Then $A=Q\otimes C_AQ$, where $C_AQ$ is the centralizer of $Q$ in $A$. Since $g^2=\lambda^2$ with $\lambda\in F^{\times}$, we may identify $Q$ with  $M_2(F)$ in such a way that $g$ is the diagonal matrix diag$(\lambda,-\lambda)$. Computations show that $\dim(g-\lambda)A=\dim(g+\lambda)A = 2\dim C_AQ$.

Conversely, suppose $\dim(g-\lambda)A=\dim(g+\lambda)A$. Let  $V_+$ and $V_-$ be  the $\lambda$-eigenspace and $-\lambda$-eigenspace of $g$ respectively. For all $u\in V$, we have $u=\frac{1}{2}(u+\lambda^{-1}g(u))+\frac{1}{2}(u-\lambda^{-1}g(u))\in V_+ +V_-$ and $V_+ \cap V_-=\{0\}$. Hence $V=V_+ \oplus V_-$. Denote by $r$ and $s$ the dimensions of $V_+$ and $V_-$ respectively. Since $g^2=\lambda^2$, $g$ is represented in $A\simeq \text{End}_D(V)$ by the diagonal matrix $\text{diag}(\lambda,\ldots,\lambda,-\lambda,\ldots,-\lambda)$ where the number of $\lambda$ is $r$ and the number of $-\lambda$ is $s$. Computations show that $\dim(g-\lambda)A=s\dim D$ and $\dim(g+\lambda)A=r\dim D$; therefore the hypothesis yields $r=s$. The endomorphism $f$ whose matrix is the block matrix  
$
f=\begin{pmatrix}
0 & 1\\ 1 & 0 
\end{pmatrix},
$
where each block is an $r\times r$ matrix, anticommutes with $g$ and is square-central. It follows that $g$ lies in the split quaternion subalgebra of $A$ generated by $g$ and $f$.

Now suppose the characteristic of $F$ is $0$. The element $g$ is represented by diag$(\lambda,\ldots,\lambda,-\lambda,\ldots,-\lambda)$. So $\Trd_A(g)=(r-s)\lambda \deg D$ and $\dim(g-\lambda)A=s\dim D$ and $\dim(g+\lambda)A=r\dim D$. Therefore $\Trd_A(g)= 0$ if and only if $\dim(g-\lambda)A=\dim(g+\lambda)A$. The proof is complete.
\end{proof}
If the characteristic of $F$ is positive, the hypothesis on the trace does not suffice as we observe in the following counterexample:
\begin{contrexample}\label{cex3.1}
Assume that $F=\mathbb F_3$, the field with three elements, and take $A=M_8(F)$. The  diagonal matrix $g= \text{diag}(1,\ldots, 1, -1)$ is such that $g^2=1$ and the trace of $g$ is $0$. But Proposition \ref{prop1.1} shows that $g$ is not in a quaternion subalgebra of $A$ since $\dim(g+1)A=56\ne\dim(g-1)A=8$.
\end{contrexample}
\begin{proposition}\label{prop1.2}
Let $A$ be a central simple $F$-algebra and let $g\in A^{\times}-F^{\times}$ be such that $g^2=a\in F^{\times}-F^{\times 2}$. The element $g$ lies in a split quaternion subalgebra of $A$ if and only if $\frac{\deg A}{\ind(A)}$ is even.
\end{proposition}
\begin{proof}
If $g\in M_2(F)\subset A$, then $A=M_2(F)\otimes C$ where $C$ is the centralizer of $M_2(F)$ in $A$. Hence, $\frac{\deg A}{\ind(A)}$ is even. Conversely,  assume $\frac{\deg A}{\ind(A)}$ is even. So, we may write $A\simeq M_2(F)\otimes A'$ for some algebra $A'$ Brauer equivalent to $A$.  Set 
$
g'=\begin{pmatrix}
0 & 1\\ 
a & 0 
\end{pmatrix}
$; we have $g'\in M_2(F)\subset A$ and $g'^2=a$. By the Skolem-Noether Theorem, $g$ and $g'$ are conjugated. It follows that $g$ is in a split quaternion subalgebra of $A$ since $g'\in M_2(F)$. 
\end{proof}
\subsection{The algebra $A$ is a division algebra}
Let $A$ be a division algebra and let $x\in A^{\times}-F^{\times}$. Recall that, $A$ being a division algebra, we have necessarily $x^2\in F^{\times}-F^{\times 2}$. Here we argue on the degree of the division algebra.
\subsection*{Degree $4$}
The following result is due to Albert and many proofs exist in the literature (see for instance \cite{Rac74}, \cite[Prop. 5.2]{LLT93}, \cite[Thm. 4.1]{Bec04}). We propose the following proof for  the reader's convenience. 
\begin{proposition}\label{prop1.3}
Suppose $A$ is a central  simple algebra over $F$ of degree $4$ and exponent $2$. Let $x\in A^{\times}-F^{\times}$ be a square-central element with $x^2\not\in F^{\times 2}$. Then, $x$ is in a quaternion $F$-subalgebra of $A$.
\end{proposition} 
\begin{proof}
Note that $F(x)$ is isomorphic to a quadratic  extension of $F$ since $x^2\in F^{\times}-F^{\times 2}$. If $A$ is not a division algebra, the result follows by Proposition \ref{prop1.2}. We assume that $A$ is a division algebra. The centralizer  $C_A(x)$ of $x$ in $A$ is a quaternion algebra over $F(x)$. The algebra $C_A(x)$ is Brauer equivalent to $A_{F(x)}$ (see for instance \cite[\S 13.3]{Pie82}). Since
\[
\text{cor}_{F(x)/F}[C_A(x)]=\cor_{F(x)/F}(\res_{F(x)/F}[A])=2[A]=0
\]
 in Br$(F)$ (see for instance \cite[(3.13)]{KMRT98}), it follows from  a result  of Albert  (see  \cite[(2.22)]{KMRT98}) that there is a quaternion algebra $Q$ over $F$ such that $C_A(x)=Q\otimes_F F(x)$. Then, $A=Q\otimes C_A(Q)$ and the centralizer $C_A(Q)$ of $Q$ is a quaternion $F$-subalgebra of $A$ containing $x$. 
\end{proof}
\subsection*{Degree $8$}
Here we  give an example of a tensor product of three quaternion algebras containing a square-central element which is in no quaternion subalgebra. This example is a private communication from Merkurjev to Tignol based on the following result:
\begin{lemma}[Tignol]\label{lm1.1}
Let $A$ be a division algebra over $F$ of degree $8$ and exponent $2$. Let $x\in A^{\times}-F^{\times}$ be such that $x^2=a\in F^{\times}$. Then, there exists quaternion algebras $Q_1, Q_2, Q_3$ such that $M_2(A)\simeq Q_1\otimes Q_2\otimes Q_3\otimes(a,y)$ for some $y\in F$.
\end{lemma}
\begin{proof}
It is shown in \cite[Thm. 5.6.38]{Jac96} that $M_2(A)$ is a tensor product of four quaternion algebras. The proof shows that one of these quaternion
algebras can be chosen to contain $x$.
\end{proof}
\begin{corollary}[Merkurjev]\label{cor3.1}
There exists a  decomposable $F$-algebra of degree $8$ and exponent $2$ containing a square-central element which is in no quaternion subalgebra.
\end{corollary}
\begin{proof}
Let $A$ be an indecomposable $F$-algebra of degree $8$ and exponent $2$ and let $x\in A$ be such that $x^2=a\in F^{\times}$ with $x\notin F$. Such an algebra $A$ exists by \cite{ART79} and the existence of such an element $x$ follows from a result of Rowen \cite[Thm. 5.6.10]{Jac96}. Lemma \ref{lm1.1}  indicates that $M_2(A)\simeq Q_1\otimes Q_2\otimes Q_3\otimes(a,y)$ for some $y\in F$. Set $D = Q_1\otimes Q_2\otimes Q_3$. We claim that $D$ is a division algebra. Indeed, if $D$ is not a division algebra then $D\simeq M_2(D')$ where $D'$ is an algebra of degree $4$ and exponent $2$. Since an exponent $2$ and degree $4$ central simple algebra is always decomposable by a well-known result of Albert (see for instance \cite{Rac74}), we deduce that $A$ is isomorphic to a product of quaternion algebras; this contradicts our hypothesis. Hence $D$ is a division algebra. Since the algebras $D_{F(\sqrt{a})}$ and $A_{F(\sqrt{a})}$ are isomorphic and $A_{F(\sqrt{a})}$ is not a division algebra, $D_{F(\sqrt{a})}$ is not a division algebra. Then, by \cite[Thm. 4.22]{Alb61} the algebra $D$ contains an element $\alpha$ such that $\alpha^2=a$ with $\alpha \notin F$. Assume that $D$ contains a quaternion subalgebra containing $\alpha$, say $(a,b)$ for some $b\in F$. The centralizer of $(a,b)$ in $D$ is an algebra of exponent $2$ and degree $4$. Thus, we have $D\simeq H_1\otimes H_2\otimes (a,b)$ where $H_i$ are quaternion algebras. It follows that  
$
M_2(A)\simeq H_1\otimes H_2\otimes (a,b)\otimes (a,y)\simeq M_2(H_1\otimes H_2\otimes (a,yb)).
$
Whence $A\simeq H_1\otimes H_2\otimes (a,yb)$; contradiction. The algebra $D$ satisfies the required conditions. 
\end{proof}
\subsection*{Degree $2^n$, $n>3$}
In this part, we generalize Corollary \ref{cor3.1}: we are going to construct a tensor product of $n$ (with $n>3$) quaternion algebras containing a square central element which is not in a quaternion subalgebra. To do this, we use valuation theory.

Let $L=F((t_1))((t_2))$ be the iterated Laurent power series field where $t_1,t_2$ are independent indeterminates over $F$ and let $D$ be a division $F$-algebra. Set $ D'=D\otimes (t_1,t_2)_L$ and let $i,j\in D'$ be such that 
$i^2=t_1,\quad j^2=t_2 \quad \text{ and } \quad ij=-ji$.  Since $i^2=t_1$ and $j^2=t_2$, every element $f\in D'$ can be written as an iterated Laurent series in $i$ and $j$ with coefficients in $D$:
\[
f=\sum_{\begin{subarray}{l}
\beta\geq n
\end{subarray}}\sum_{\alpha\geq m_\beta}d_{\alpha,\beta}i^{\alpha}j^{\beta} \quad \text{ with } d_{\alpha,\beta}\in D \text{ and } n, m_\beta\in\Z.
\]
Define  $v\, :\, D'^{\times}\, \longrightarrow\, (\frac{1}{2}\Z)^2$ (where $(\frac{1}{2}\Z)^2$ is ordered lexicographically from right-to-left) by 
\[
v(f)=\inf\Big\{(\frac{\alpha}{2},\frac{\beta}{2})\,|\,\, d_{\alpha,\beta}\ne 0\Big\}.
\]
Computations show that $v$ is a valuation on $D'$. Actually, $v$ is the unique extension of the $(t_1,t_2)$-adic valuation on $L$ (which is Henselian). As in the  previous section, for $f\in D'^{\times}$, the leading term of $f$ is defined to be
\[
\ell(f)=d_{m,n}i^mj^n \quad \text{ where }\quad (m,n)= v(f).
\]
Straightforward computations show that
\begin{itemize}
\item[(i)] $\ell(fg)=\ell(f)\ell(g)$, for $f,g\in D'^{\times}$;
\item[(ii)] $\ell(d)=d$, for $d\in D^{\times}$;
\item[(iii)] $\ell(z)\in L $, for $z\in L^{\times}$.
\end{itemize}
We have the following generalization of Corollary \ref{cor3.1}:
\begin{proposition}\label{prop1.4}
Let $D$ be a division algebra over $F$. Let $x\in D^{\times}-F$ be a square-central element which is in no quaternion subalgebra of $D$. Then $D\otimes (t_1,t_2)_L$ has no quaternion subalgebra containing $x$. 
\end{proposition}
\begin{proof} Suppose there is $y\in D\otimes (t_1,t_2)_L$ such that $y^2\in L^\times$ and $xy=-yx$. Let $\ell(y)= d i^\alpha j^\beta$ with $d\in D^\times$ and $\alpha,\beta\in\Z$. We have $\ell(y)i^{-\alpha}j^{-\beta}=d$ and $d^2\in F\m$. Since $xy=-yx$ we have $\ell(x)\ell(y)=-\ell(y)\ell(x)$, that is, $x\ell(y)=-\ell(y)x$. Hence, $d$ anticommutes with $x$; contradiction with the choice of $x$. 
\end{proof}
\subsection{An application}
Corollary \ref{cor3.1} implies that if $F$ is the center of an indecomposable  algebra of degree $8$ and exponent $2$, then there exist a decomposable division algebra of degree $8$ and exponent $2$ containing a square-central element which is not in a quaternion subalgebra. Conversely, let $D$ be a division algebra of degree $8$ and exponent $2$  over $F$ and let $F(\sqrt{a})\subset D$ be a quadratic field extension of $F$ such that $F(\sqrt{a})$ is not in a quaternion subalgebra of $D$ (the algebra $D$ could be decomposable). Theorem \ref{thm3.2} shows that the division algebra Brauer equivalent to $D\otimes_F(a,t)_{F(t)}$ is an indecomposable algebra of degree $8$ and exponent $2$ over $F(t)$.

As an application, we are going now to give an example of indecomposable algebra of degree $8$ and exponent $2$ over a field of $2$-cohomological dimension $4$. Let $F$ be a field of characteristic different from $2$ and let us denote  $K=F(\sqrt{a})$. Let $B$ be a biquaternion algebra over $K$ with trivial corestriction, $\cor_{K/F}(B)=0$. In \cite{Bar}, it is associated with $B$ a degree three cohomological invariant $\delta_{K/F}(B)$ with value in $H^3(F,\mu_2)/\cor_{K/F}((K^\times)\cdot [B])$ where $H^3(F,\mu_2)$ is the third Galois cohomology group of $F$ with coefficients in $\mu_2=\{\pm 1\}$. It is also shown in \cite{Bar} that $B$ has a descent to $F$ (that is, $B=B_0\otimes_F K$ for some biquaternion algebra $B_0$ defined over $F$)  if and only if $\delta_{K/F}(B)=0$.

Now, let $D$ be a central simple algebra of degree $8$ and exponent $2$ over $F$ containing $K$ such that $K$ is not in a quaternion subalgebra of $D$. That also means $\delta_{K/F}(C_DK)\ne 0$ where $C_DK$ denotes the centralizer of $K$ in $D$. Consider the following \emph{Merkurjev extension} $\mathbb{M}$ of $F$:
\[
F=F_0\subset F_1\subset\cdots \subset F_{\infty}=\bigcup_iF_i =:\mathbb{M}
\]
where the field $F_{2i+1}$ is the maximal odd degree extension of $F_{2i}$; the field $F_{2i+2}$ is the composite of all the function fields  $F_{2i+1}(\pi)$, where $\pi$ ranges over all $4$-fold Pfister forms over $F_{2i+1}$. The arguments used by  Merkurjev in \cite{Mer92} show that the $2$-cohomological dimension $cd_2(\mathbb{M})\le 3$.  We have the following result:
\begin{theorem}\label{application}
The algebra $D$ and the field $\mathbb{M}$ are as above. The division algebra Brauer equivalent to
\[
D_{\mathbb M}\otimes_{\mathbb M}(a,t)_{\mathbb M(t)}
\]
is indecomposable of degree $8$ and exponent $2$ over $\mathbb M(t)$, where $t$ is an indeterminate. 
\end{theorem}
\begin{proof}
Put $B=C_DK$. As observed in the proof of \cite[Thm. 1.3]{Bar} the $2$-cohomological dimension of $\mathbb{M}$ is exactly $3$. It follows from \cite[Prop. 4.7]{Bar} and a result of Merkurjev (see Theorem A.9 of \cite{Bar}) that the scalar extension map
\[
\frac{H^3(F,\mu_2)}{\cor_{K/F}((K^\times)\cdot [B])}\longrightarrow \frac{H^3(\mathbb M,\mu_2)}{\cor_{\mathbb M(\sqrt{a})/\mathbb M}(\mathbb M(\sqrt{a})^\times\cdot [B_{\mathbb M(\sqrt{a})}])}
\]
is an injection. So, $\delta_{\mathbb{M}(\sqrt{a})/\mathbb{M}}(B)\ne 0$ since $\delta_{K/F}(B)\ne 0$. Hence the extension  $\mathbb M(\sqrt{a})$ is not in a quaternion subalgebra of $D_{\mathbb M}$. Therefore the division algebra Brauer equivalent to $D_{\mathbb M}\otimes_{\mathbb M}(a,t)_{\mathbb M(t)}$ is an indecomposable algebra of degree $8$ and exponent $2$ over $\mathbb M(t)$ by Corollary \ref{thm3.2}; as desired.  
\end{proof}

\bigskip
\subsection*{Acknowledgements}
This work is a generalization of some results of my PhD thesis. I gracefully thank  my advisors A. Qu\'eguiner-Mathieu and J.-P. Tignol  for their support, ideas and for sharing their  knowledge. I thank A. S. Merkurjev for providing Corollary \ref{cor3.1} (a private communication to J.-P. Tignol).


\end{document}